\documentclass[12pt,a4paper]{amsart}
\usepackage{amssymb,latexsym,amsmath}
\usepackage{tikz}
\usetikzlibrary{matrix,arrows,positioning,calc}
\usepackage{enumerate}
\usepackage{multirow}

\usepackage[T1]{fontenc}

\newtheorem{thm}{Theorem}
\newtheorem{cor}[thm]{Corollary}
\newtheorem{lem}[thm]{Lemma}
\newtheorem{prop}[thm]{Proposition}

\newcommand{\R}{\mathbb{R}}
\newcommand{\C}{\mathbb{C}}

\newcommand{\Aj}{{\mathcal{A}_{1,-\alpha}}}
\newcommand{\Ad}{{\mathcal{A}_{d,-\alpha}}}

\DeclareMathOperator{\re}{Re}

\begin{document}

\title{Fractional calculus for power functions}

\author{Bart{\l}omiej Dyda}

\address{
Faculty of Mathematics\\ University of Bielefeld\\
Postfach 10 01 31,
D-33501 Bielefeld, Germany\\
Tel.: +49 (0)521-1062429\\
Fax: +49 (0)521 106-89027\\
\and
 Institute of Mathematics and Computer Science\\ Wroc{\l}aw University of Technology\\
Wybrze\.ze Wyspia\'nskiego 27,
50-370 Wroc{\l}aw, Poland\\
Tel.:+48 71 320-3182\\
Fax: +48 71 328-0751\\
}
\email{bdyda (at) pwr wroc pl}

\begin{abstract}
We calculate the fractional Laplacian for functions
of the form $u(x)=(1-|x|^2)_+^p$ and $v(x)=x_d u(x)$.
As an application, we estimate the first eigenvalues of the
fractional Laplacian in a ball.
\end{abstract}


\maketitle

\section{Main result and discussion}\label{sec:i}
We consider the fractional Laplacian  $\Delta^{\alpha/2}$
 defined by
\begin{equation}\label{eq:fraclapl}
 \Delta^{\alpha/2} u(x) = \Ad \lim_{\varepsilon\to 0^+}
\int_{\R^d \cap \{|y-x|>\varepsilon\}}\frac{u(y)-u(x)}{|x-y|^{d+\alpha}} \,dy.
\end{equation}
Here $\Ad = \frac{2^\alpha \Gamma(\frac{\alpha+d}{2})}{ \pi^{d/2} |\Gamma(-\frac{\alpha}{2})|}$.
This is an important operator for probability and analysis
\cite{MR2569321}, \cite{Landkof}.
The main results of this paper are the following formulae for the fractional Laplacian
of power functions.

\begin{thm}\label{thm:up}
Let $d\geq 1$, $0<\alpha<2$ and $p>-1$. Define
\begin{align}
 u_p^{(d)}(x) &= (1-|x|^2)_+^p, \quad x\in \R^d,\label{upn}\\
 v_p^{(d)}(x) &= (1-|x|^2)_+^p x_d, \quad x\in \R^d,\label{vpn}\\
 \Phi^{(d)}_{p,\alpha}(x) &=  \frac{\Ad B(-\frac{\alpha}{2}, p+1)\pi^{d/2}}{\Gamma(\frac{d}{2})} \;\;
 {} _2F_1\Big(\frac{\alpha+d}{2}, -p+\frac{\alpha}{2}; \frac{d}{2}; x \Big).
\end{align}
If $x\in \R^d$ and $|x|<1$, then
\begin{align}
 \Delta^{\alpha/2} u_p^{(d)}(x) &=\Phi^{(d)}_{p,\alpha}(|x|^2),\label{Deltaupn}\\
 \Delta^{\alpha/2} v_p^{(d)}(x) &= x_d  \Phi^{(d+2)}_{p,\alpha}(|x|^2).\label{Deltavpn}
\end{align}
\end{thm}

By $B$ and $\Gamma$ we denote Euler's beta and gamma functions,
respectively, and $_2F_1$ is Gauss' hypergeometric function \cite{Erdelyi}.
Some special cases of the above theorem were known before. A calculation
of $\Delta^{\alpha/2} u_{\frac{\alpha}{2}}^{(d)}$ was done by Getoor \cite{MR0137148}, and 
explicit expression for
$\Delta^{\alpha/2} u_{\frac{\alpha}{2}-1}^{(d)}$ and $\Delta^{\alpha/2} v_{\frac{\alpha}{2}-1}^{(d)}$
 may be deduced from \cite{Hmissi}, see also \cite{BogdanMartin}.
A formula for $\Delta^{\alpha/2} u_p^{(d)}$ with $p\in (0,1)$ and $\alpha\in (-1,1)$ is announced
in \cite{BilerImbertKarch}. Finally, (\ref{Deltaupn}) was given in \cite{DydaHardy1dim}
 for $p=\frac{\alpha-1}{2}$ and $d=1$, where it was used to obtain
a fractional Hardy inequality with a remainder term.
For similar results on fractional derivatives we refer the reader to \cite{MR1347689}.

We note that if $p=n+\alpha/2$ and $n$ is a nonnegative integer,
then  $\Delta^{\alpha/2} u_p^{(d)}$ and $ \Delta^{\alpha/2} v_p^{(d)}$ are
polynomials of degree $n$ and $n+1$, respectively. 
For $d=1$, every polynomial can be expressed as a linear combination of polynomials
of the form $(1-x^2)^j$ and $x(1-x^2)^k$. Hence, if $P$ is \emph{any} polynomial of degree $n$, then
$\Delta^{\alpha/2}(P(x)(1-x^2)^{\alpha/2})$, for $x\in (-1,1)$, is a polynomial of degree $n$, too.
This also trivially holds for $\alpha=2$.
This may allow one for the development of an explicit integro--differential calculus for $\Delta^{\alpha/2}(\cdot 1_{|x|<1})$,
which is one of the motivations for this work.
It is noteworthy that $u_p^{(d)}(x) =  {} _2F_1\Big(\frac{d}{2}, -p; \frac{d}{2}; |x|^2 \Big)$ for~\mbox{$|x|<1$},
therefore  (\ref{Deltaupn}) may be rewritten in the following elegant form
\begin{equation}
\Delta^{\alpha/2}\left({} _2F_1\Big(\frac{d}{2}, -p; \frac{d}{2}; |x|^2 \Big) 1_{|x|<1}\right) =
C_{d,\alpha,p}\;\;
 {} _2F_1\Big(\frac{d}{2} + \frac{\alpha}{2}, -p+\frac{\alpha}{2}; \frac{d}{2}; |x|^2 \Big),
\end{equation}
where $x\in\R^d$, $|x|<1$, $C_{d,\alpha,p}={\Ad B(-\frac{\alpha}{2}, p+1)\pi^{d/2}}/{\Gamma(\frac{d}{2})}$,
$d=1,2,\ldots$ and $\alpha\in (0,2)$.

Let $B=B_1^{(d)}=\{x\in \R^d:|x|<1\}$ be the unit ball in $\R^d$.
We consider the quadratic form
\[
 \mathcal{E}(u)=\frac{\Ad}{2}
\int_{\R^d\times \R^d} 
\frac{(u(x)-u(y))^2}{|x-y|^{d+\alpha}} \,dx\,dy,
\]
with the domain
\[
 D(\mathcal{E})=\{u\in L^2(\R^d): \mathcal{E}(u)<\infty,\; \textrm{$u=0$ on $B^c$}\},
\]
and the corresponding symmetric bilinear form $\mathcal{E}(\cdot, \cdot)$.
Then $(\mathcal{E}(\cdot, \cdot), D(\mathcal{E}))$ is the Dirichlet form of the $\alpha$-stable, rotation
invariant L\'evy process killed upon leaving $B$.
To point out applications of Theorem~\ref{thm:up}, we consider the  spectral problem
of finding $\phi\in D(\mathcal{E})$ such that
\begin{equation}\label{spprob}
 \mathcal{E}(\phi,g) =  \lambda \int \phi(x) g(x)\,dx, \quad g \in D(\mathcal{E}).
\end{equation}
It is known \cite{Getoor1959a} that there
exists an orthonormal basis of $L^2(B)\subset L^2(\R^d)$, consisting of  eigenfunctions $\phi_1$,  $\phi_2$, $\phi_3$, \ldots
  with the corresponding eigenvalues 
$0<\lambda_1<\lambda_2\leq \lambda_2 \leq \ldots$.
The latter means that $\phi=\phi_n$ and $\lambda=\lambda_n$ satisfy (\ref{spprob}).
We note that these eigenfunctions are in the domain of the generator, and we have that
$\Delta^{\alpha/2}\phi_n = -\lambda_n \phi_n$ on $B$  in the sense of definition (\ref{eq:fraclapl}),
see also \cite{MR1825645}. When there is a risk of confusion, we write the dimension of the underlying
space in superscripts, i.e., we write $\lambda_n^{(d)}$ for $\lambda_n$ and $\phi_n^{(d)}$ for $\phi_n$.

The eigenproblem is the main motivation for this research.
We should note that
the eigenvalues $\lambda_1$, $\lambda_2$, \ldots are not known explicitly even in the case of $d=1$ and $B=(-1,1)$.
 A number of methods to study this one-dimensional case, and more general cases, were developed by 
several authors
\cite{MR2056835},  \cite{MR2214145}, \cite{MR2273977}, \cite{MR2217951}, \cite{MR1840105},
\cite{Getoor1959b}, \cite{MR2158176}, \cite{MR2078980}, \cite{MR0147937}, \cite{MR0212621}, \cite{MR2679702}, 
\cite{KwasnickiInterval}, \cite{MR2445510}, \cite{MR2534231}, \cite{MR1119514}, \cite{MR2139213}.
The symmetry of eigenfunctions plays an important role in these investigations. 
We also note that this spectral problem
may be formulated in terms of the isotropic $\alpha$-stable L\'evy process in $\R^d$.
For details we refer the reader to \cite{MR2322501}. 

 Let $\lambda_*$ be the smallest number such that there
exists an  eigenfunction $\phi_*$ which is antisymmetric, $\phi_*(-x)=-\phi_*(x)$, and has eigenvalue $\lambda_*$.
It is conjectured, but not yet proved in the full range of $\alpha\in (0,2)$ and $d$, that $\lambda_*=\lambda_2$.
In the classical case ($\alpha=2$), and also in the $1$--dimensional case
 for $\alpha\geq 1$ \cite{MR2679702}, we do have $\lambda_*=\lambda_2$.
It is natural to ask whether  $\lambda_* = \lambda_2$.
While we do not answer this question here, the calculus of power functions given by
Theorem~\ref{thm:up} may be used to investigate the eigenfunctions of the fractional Laplacian
 in the ball in this and related problems.

We should note that there always exists an antisymmetric eigenfunction.
Indeed, there exists a non-symmetric eigenfunction $\phi$, and we may
 let $\tilde{\phi}(x)=\phi(x)-\phi(-x)$, which is an antisymmetric
eigenfunction with the same eigenvalue as $\phi$.
Similarly  $\phi_*$ may and will be assumed  antisymmetric
with respect to the last coordinate axis.

The similarity between \eqref{Deltaupn} and \eqref{Deltavpn} leads us to the following result.
\begin{prop}\label{prop:lambdaeq}
$\lambda_*^{(d)} = \lambda_1^{(d+2)}$.
\end{prop}

Noteworthy, for $p=n+\alpha/2$  the right-hand sides of  (\ref{Deltaupn}) and (\ref{Deltavpn}) are
polynomials, which  gives an efficient way to explicitly estimate
$\lambda_1$ and $\lambda_*$ using some versions of Barta inequality \cite{MR2139213}, see Sections~\ref{sec:lower} and \ref{sec:upper}.
In particular, we obtain the following corollary.
\begin{cor}\label{cor:est}
Let
\begin{equation}\label{eq:mu}
 \mu_{d,\alpha} =  \frac{2^\alpha \Gamma(\frac{\alpha}{2}+1) \Gamma(\frac{\alpha+d}{2}) (\alpha+2)(\alpha+d) (6-\alpha)}
         { \Gamma(\frac{d}{2})(12d+(16-2d)\alpha)}.
\end{equation}
We have
\begin{equation}
\lambda_1 \geq \mu_{d,\alpha}
\end{equation}
and
\begin{equation}
  \lambda_* \geq  \mu_{d+2,\alpha}.
\end{equation}
\end{cor}
We note that the above result improves estimates from
\cite{MR1119514} and \cite[Corollary 1]{MR1840105} by a~factor $(\alpha+2)(\alpha+1)(6-\alpha)/(12+14\alpha) > 1$,
and it also improves some of the estimates from \cite{KwasnickiInterval}.

The paper is structured as follows. In Section~\ref{sec:dim1} we prove Theorem~\ref{thm:up}
for $d=1$, then in Section~\ref{sec:dimd} we prove the $d$-dimensional case.
In Sections~\ref{sec:lower} and \ref{sec:upper} we prove Proposition~\ref{prop:lambdaeq} and
we derive lower and upper bounds for the  eigenvalues
$\lambda_1$ and~$\lambda_*$. 

\section{One--dimensional case}\label{sec:dim1}
The following technical result is the first step in the proof of Theorem~\ref{thm:up}.

\begin{lem}\label{lem:pv}
If $p>-1$, $0<\alpha<2$ and $x\in (-1,1)$, then
\begin{align}
I_m(p) &:=  {p.v.}\int_{-1}^1 \frac{(1-tx)^{\alpha-m-2p}-1}{|t|^{1+\alpha}}\;
 (1-t^2)^p\,dt  \nonumber\\
&= B\Big(-\frac{\alpha}{2}, p+1\Big)
  \left(  {} _2F_1\Big(-\frac{\alpha}{2}, p+m-\frac{1}{2}-\frac{\alpha}{2}; \frac{1}{2}; x^2\Big)
 -1 \right),\label{eq:pv}
\end{align}
where $m=1$ or $m=2$, and $p.v.$ means the Cauchy principal value.
\end{lem}
\begin{proof}
We assume that
$p\neq\frac{\alpha-2}{2}$, since otherwise the beta function on the right-hand side of (\ref{eq:pv}) is zero
and the result is obvious.
We have
\begin{align*}
I_m(p) &=
  \, {p.v.}\int_{-1}^1 \frac{(1-tx)^{\alpha-m-2p}-1}{|t|^{1+\alpha}}\;
 (1-t^2)^p\,dt \\
&=
  \int_{-1}^1 \sum_{k=2}^\infty \frac{(2p+m-\alpha)_k (tx)^k}{ k! |t|^{1+\alpha}}\;
 (1-t^2)^p\,dt \\
&=
  2 \int_{0}^1  \sum_{k=1}^\infty \frac{ (2p+m-\alpha)_{2k} (tx)^{2k}}{ (2k)! |t|^{1+\alpha}}\;
 (1-t^2)^p\,dt \\
&=\bigg(\sum_{k=0}^\infty \frac{ (2p+m-\alpha)_{2k} B\Big(k-\frac{\alpha}{2}, p+1\Big)}{ (2k)!} x^{2k}\bigg) - 
 B\Big(-\frac{\alpha}{2}, p+1\Big)\\
&=:S_m - B\Big(-\frac{\alpha}{2}, p+1\Big).
\end{align*}
Here $(a)_n$ is the Pochhammer symbol, that is, $(a)_0=1$ and $(a)_n = a(a+1)\ldots (a+n-1)$ for $n=1,2,\ldots$.
We observe that 
\begin{equation}\label{eq:wz1}
 (2p+m-\alpha)_{2k} = 2^{2k} \left(p+\frac{m}{2}-\frac{\alpha}{2}\right)_k \left(p+\frac{m+1}{2}-\frac{\alpha}{2}\right)_k,
\end{equation}
and, by the doubling formula,
\begin{equation}\label{doubling}
 \Gamma(2x)=\Gamma(x)\Gamma(x+1/2)2^{2x-1}/\Gamma(1/2),
\end{equation}
applied to $2x=2k+1$, we have
\begin{equation}\label{eq:wz2}
(2k)! = 2^{2k} {\textstyle \left(\frac{1}{2}\right)\!_k }k!.
\end{equation}
We also have,
\begin{equation}\label{eq:wz3}
B\Big(k-\frac{\alpha}{2}, p+1\Big) = \frac{\left(-\frac{\alpha}{2}\right)_k }{ \left(p+1-\frac{\alpha}{2}\right)_k}B\Big(-\frac{\alpha}{2}, p+1\Big).
\end{equation}
Thus, by (\ref{eq:wz1}), (\ref{eq:wz2}) and (\ref{eq:wz3}) we obtain
\begin{align*}
S_m&= B\Big(-\frac{\alpha}{2}, p+1\Big)
 \sum_{k=0}^\infty \frac{\left(-\frac{\alpha}{2}\right)_k \left(p+\frac{m}{2}-\frac{\alpha}{2}\right)_k \left(p+\frac{m+1}{2}-\frac{\alpha}{2}\right)_k 
}{ \left(p+1-\frac{\alpha}{2}\right)_k  \left(\frac{1}{2}\right)_k k! }\, x^{2k}.
\end{align*}
For $m=1$ or $m=2$ the factor $\left(p+1-\frac{\alpha}{2}\right)_k$ in the denominator cancels with one of the terms
in the numerator, and the result follows.
\end{proof}

In the following lemma we prove (\ref{Deltaupn}) for $d=1$.
Recall that
\begin{equation}\label{up}
 u_p(x) = (1-x^2)_+^p,\quad x\in\R.
\end{equation}

\begin{lem}\label{lem:Deltaup}
If $0<\alpha<2$, $p>-1$  and $x\in (-1,1)$ then
\begin{align}\label{Deltaup}
 \Delta^{\alpha/2} u_p(x) &=\Aj B\Big(-\frac{\alpha}{2}, p+1\Big)
  {} _2F_1\Big(-\frac{\alpha}{2}, p+\frac{1}{2}-\frac{\alpha}{2}; \frac{1}{2}; x^2\Big)
(1-x^2)^{p-\alpha}\\
&=\Aj B\Big(-\frac{\alpha}{2}, p+1\Big)
 {} _2F_1\Big(\frac{\alpha+1}{2}, -p+\frac{\alpha}{2}; \frac{1}{2}; x^2\Big).\label{Deltaup2}
\end{align}
\end{lem}

\begin{proof}
We recall from \cite[Lemma 2.1]{DydaHardy1dim} the following formula
\begin{align}
 L u_p(x) &=
\frac{(1-x^2)^{p-\alpha}}{\alpha} \Bigg(
(1 - x)^\alpha + (1 + x)^\alpha
 \nonumber\\
&
- (2p+2-\alpha)B(p+1,1-\alpha/2)
+\alpha I_1(p) 
\Bigg), \label{Lup}
\end{align}
where $I_1(p)$ is given by (\ref{eq:pv}).

By $(2p+2-\alpha)B(p+1,1-\alpha/2) = - \alpha B(p+1,-\alpha/2)$ and  Lemma~\ref{lem:pv},
\begin{align*}
\alpha I_1(p) - (2p+2-\alpha)&B(p+1,1-\alpha/2)\\
 &=\alpha B\Big(-\frac{\alpha}{2}, p+1\Big) {} _2F_1\Big(-\frac{\alpha}{2}, p+\frac{1}{2}-\frac{\alpha}{2}; \frac{1}{2}; x^2\Big).
\end{align*}
This proves (\ref{Deltaup}). The formula (\ref{Deltaup2}) follows from \cite[formula 2.9(2), page 105]{Erdelyi}.
\end{proof}

\begin{table}[ht]
\caption{The fractional Laplacian for some functions vanishing outside of $(-1,1)$.}
\centering
\begin{tabular}{c c}
\hline\hline
$u(x)$ on $(-1,1)$ & $\Delta^{\alpha/2} u(x)$ for  $x\in(-1,1)$ \\ [0.5ex]
\hline
$(1-x^2)^{-1+\alpha/2}$ & $0$ \\
$(1-x^2)^{\alpha/2}$ & $-\Gamma(\alpha+1)$ \\
$(1-x^2)^{1+\alpha/2}$ & $-\Gamma(\alpha+1) \frac{\alpha+2}{2}(1-(1+\alpha)x^2)$ \\
$(1-x^2)^{2+\alpha/2}$ & $-\Gamma(\alpha+1)\frac{\alpha+2}{2}\frac{\alpha+4}{4} 
   \Big( 1-(2\alpha+2)x^2 + (\frac{\alpha}{3}+1)(\alpha+1)x^4 \Big)$\\
\hline\hline
\end{tabular}
\label{table:values}
\end{table}

In the following lemma we prove (\ref{Deltavpn}) for $d=1$.

\begin{lem}\label{lem:Deltavp}
 Let $p>-1$ and $v_p(x) = x(1-x^2)_+^p$ for $x\in \R$.
For $0<\alpha<2$ and $x\in (-1,1)$ we have
\begin{align}
 \Delta^{\alpha/2} v_p(x) &=\Aj B\Big(-\frac{\alpha}{2}, p+1\Big) (\alpha+1)\label{Deltavp}\\
&\quad \times  {} _2F_1\Big(-\frac{\alpha}{2}, p+\frac{3}{2}-\frac{\alpha}{2}; \frac{3}{2}; x^2\Big)
x(1-x^2)^{p-\alpha}\nonumber\\
&=\Aj B\Big(-\frac{\alpha}{2}, p+1\Big) (\alpha+1) \;\;
{} _2F_1\Big(\frac{\alpha+3}{2}, -p+\frac{\alpha}{2}; \frac{3}{2}; x^2\Big)x.\label{Deltavp2}
\end{align}
\end{lem}

\begin{proof}
We write
\begin{align*}
\mathcal{A}^{-1}_{1,-\alpha}&\Delta^{\alpha/2} v_p(x) \\
&=p.v. \int_{-1}^1 \frac{y(1-y^2)^p-x(1-x^2)^p}{|y-x|^{1+\alpha}}\,dy
- v_p(x) \int_{\R\setminus(-1,1)} \frac{dy}{|y-x|^{1+\alpha}}\\
&=p.v. \int_{-1}^1 \frac{y(1-y^2)^p-x(1-x^2)^p}{|y-x|^{1+\alpha}}\,dy
- \frac{v_p(x)}{\alpha} \left(\frac{1}{(x+1)^\alpha} + \frac{1}{(1-x)^\alpha} \right) \\
&=: I - \frac{v_p(x)}{\alpha} \left(\frac{1}{(x+1)^\alpha} + \frac{1}{(1-x)^\alpha} \right).
\end{align*}
To evaluate $I$,  we change the variable to $t=\frac{x-y}{1-xy}$, see \cite[the proof of Lemma~2.1]{DydaHardy1dim}
for more details. We obtain
\begin{align*}
I &=
(1-x^2)^{p-\alpha}
 p.v.\int_{-1}^1 \frac{(1-t^2)^p (x-t) - x(1-tx)^{2p+1}}{|t|^{1+\alpha}}\;
 (1-tx)^{\alpha-2-2p}\,dt\\
&=
 (1-x^2)^{p-\alpha} \Bigg[ 
  x p.v.\int_{-1}^1 \frac{(1-tx)^{\alpha-2p-2} - 1}{|t|^{1+\alpha}}(1-t^2)^p\,dt\\
& + x p.v.\int_{-1}^1 \frac{(1-t^2)^p - 1}{|t|^{1+\alpha}}\,dt
 + x p.v.\int_{-1}^1 \frac{1 - (1-tx)^{\alpha-1}}{|t|^{1+\alpha}}\,dt\\
& - p.v.\int_{-1}^1 \frac{(1-t^2)^p t(1-tx)^{\alpha-2p-2}}{|t|^{1+\alpha}}\,dt \Bigg].
\end{align*}
We have by \cite[the proof of Lemma~2.1]{DydaHardy1dim}
\begin{align*}
p.v.\int_{-1}^1 \frac{(1-t^2)^p - 1}{|t|^{1+\alpha}}\,dt &= \frac{2}{\alpha}\left[1-(p+1-\alpha/2)B(p+1,1-\alpha/2)\right],\\
p.v.\int_{-1}^1 \frac{1 - (1-tx)^{\alpha-1}}{|t|^{1+\alpha}}\,dt &=
  \frac{1}{\alpha}\Big( 1 - x \Big)^\alpha + \frac{1}{\alpha}\Big( 1 + x \Big)^\alpha -  \frac{2}{\alpha}.
\end{align*}
By Lemma~\ref{lem:pv} we obtain
\begin{align*}
p.v.\int_{-1}^1 &\frac{(1-tx)^{\alpha-2p-2} - 1}{|t|^{1+\alpha}}(1-t^2)^p\,dt 
=I_2(p)\\
& = B\Big(-\frac{\alpha}{2}, p+1\Big) 
\left( {} _2F_1\Big(-\frac{\alpha}{2}, p+\frac{3}{2}-\frac{\alpha}{2}; \frac{1}{2}; x^2\Big) - 1\right).
\end{align*}
For $p\neq\frac{\alpha-2}{2}$ we have
\begin{align*}
 K&:= p.v.\int_{-1}^1 \frac{(1-t^2)^p t(1-tx)^{\alpha-2p-2}}{|t|^{1+\alpha}}\,dt \\
&=p.v.\int_{-1}^1 \frac{(1-t^2)^p t}{|t|^{1+\alpha}}\sum_{k=0}^\infty \frac{(2p+2-\alpha)_k}{k!} (tx)^k \,dt \\
&=\sum_{k=0}^\infty 2 \int_0^1 \frac{(1-t^2)^p t}{|t|^{1+\alpha}} \frac{(2p+2-\alpha)_{2k+1}}{ (2k+1)!} (tx)^{2k+1} \,dt \\
&=\sum_{k=0}^\infty B(p+1,\frac{2k+2-\alpha}{2}) \frac{(2p+2-\alpha)_{2k+1}}{ (2k+1)!} x^{2k+1}.
\end{align*}
Using the  doubling formula (\ref{doubling})  for $x=k+1$, we obtain
\[
 (2k+1)! = 2^{2k+1} {\textstyle \left(\frac{1}{2}\right)\!_{k+1} }k!.
\]
We also have
\[
 (2p+2-\alpha)_{2k+1} = 2^{2k+1} (p+1-\frac{\alpha}{2})_{k+1} (p+\frac{3}{2}-\frac{\alpha}{2})_{k}\;,
\]
and
\[
 B(p+1,\frac{2k+2-\alpha}{2}) = B(-\frac{\alpha}{2}, p+1) \frac{ (-\frac{\alpha}{2})_{k+1}}{ (p+1-\frac{\alpha}{2})_{k+1}}.
\]
Hence
\begin{align*}
 K&=  B(-\frac{\alpha}{2}, p+1) \sum_{k=0}^\infty 
  \frac{ (-\frac{\alpha}{2})_{k+1} (p+\frac{3}{2}-\frac{\alpha}{2})_{k}}
    { \left(\frac{1}{2}\right)\!_{k+1} k! } x^{2k+1}.\\
 &= -\alpha B(-\frac{\alpha}{2}, p+1) \sum_{k=0}^\infty 
  \frac{ (1-\frac{\alpha}{2})_{k} (p+\frac{3}{2}-\frac{\alpha}{2})_{k}}
    { \left(\frac{3}{2}\right)\!_{k} k! } x^{2k+1}\\
&=-\alpha B\Big(-\frac{\alpha}{2}, p+1\Big) x \cdot {}_2F_1\Big(1-\frac{\alpha}{2}, p+\frac{3}{2}-\frac{\alpha}{2}; \frac{3}{2}; x^2\Big).
\end{align*}
This holds also for $p=\frac{\alpha-2}{2}$, since in this case we have $K=0$.
Thus,
\begin{align*}
 \Delta^{\alpha/2} v_p(x)& =\mathcal{A}_{1,-\alpha}
B\Big(-\frac{\alpha}{2}, p+1\Big) (1-x^2)^{p-\alpha} x  \\
&\quad \times
   \left( {} _2F_1\Big(-\frac{\alpha}{2}, p+\frac{3}{2}-\frac{\alpha}{2}; \frac{1}{2}; x^2\Big)
 + \alpha \;\; {}_2F_1\Big(1-\frac{\alpha}{2}, p+\frac{3}{2}-\frac{\alpha}{2}; \frac{3}{2}; x^2\Big) \right).
\end{align*}
Formula (\ref{Deltavp}) follows by \cite[formula 2.8(35), page 103]{Erdelyi},
and (\ref{Deltavp2}) is then a consequence of \cite[formula 2.9(2), page 105]{Erdelyi}.
\end{proof}

In Table~\ref{table:valuesvp} we list the  fractional Laplacian of a few functions $v_p$ .
We note that the first example in Table~\ref{table:valuesvp} may be considered
a~linear combination of Martin kernels of the interval, see \cite[(89)]{MR2365478} and the references given there.
\begin{table}[ht]
\caption{The fractional Laplacian for some functions, when defined as zero outside of $(-1,1)$.}
\centering
\begin{tabular}{c c}
\hline\hline
$v(x)$ on $(-1,1)$ & $\Delta^{\alpha/2} v(x)$ for $x\in(-1,1)$ \\ [0.5ex]
\hline
$x(1-x^2)^{-1+\alpha/2}$ & $0$ \\
$x(1-x^2)^{\alpha/2}$ & $-\Gamma(\alpha+2)x$ \\
$x(1-x^2)^{1+\alpha/2}$ & $- \frac{\Gamma(\alpha+3)}{6}(3-(3+\alpha)x^2)x$ \\
$x(1-x^2)^{2+\alpha/2}$ & $- \frac{\Gamma(\alpha+3)(\alpha+4)}{120} 
   \Big( 15-(10\alpha+30)x^2 + (\alpha +3)(\alpha+5)x^4 \Big)x$\\
\hline\hline
\end{tabular}
\label{table:valuesvp}
\end{table}

\section{Multi-dimensional case}\label{sec:dimd}
We recall the notation (\ref{upn}) and note that $u_p^{(d)}(x) = u_p(|x|)$ for $x\in \R^d$,
with $u_p$ given by (\ref{up}).
We let $S^{(d-1)} = \{x\in \R^d:|x|=1\}$, the unit sphere in $\R^d$.

\begin{lem}\label{lem:Deltaradial}
Let $d\geq 2$, $0<\alpha<2$ and $p>-1$. If $x\in \R^d$ and $|x|<1$, then
\begin{equation}\label{eq:Deltaradial}
 \Delta^{\alpha/2} u_p^{(d)}(x) =
 \frac{\Ad}{2\Aj} \int_{S^{d-1}} (1- |x|^2 + |h_dx|^2)^{p-\alpha/2}
    \Delta^{\alpha/2} u_p\bigg(  \frac{|x|h_d}{\sqrt{1- |x|^2 + |h_dx|^2}}  \bigg) \, dh.
\end{equation}
\end{lem}

\begin{proof}
Without loss of generality, we may assume that  $x = (0,0,\ldots,0,|x|)$.
For $|x|<1$ we have,
\begin{align*}
 \Delta^{\alpha/2}u_p^{(d)}(x)  &= \Ad p.v. \int_{\R^d} \frac{u_p^{(d)}(y) - (1-|x|^2)^p}
     {|x-y|^{d+\alpha}}\,dy\\
&=\frac{\Ad}{2} \int_{S^{d-1}} dh
\; p.v.\int_{\R}
\frac{ u_p^{(d)}(x+h t) - (1-|x|^2)^p}{|t|^{1+\alpha}} \,dt.
\end{align*}
We calculate the (inner) principal value integral by changing the variable
$t=-|x|h_d + s \sqrt{|h_dx|^2-|x|^2+1}$. We obtain
\begin{align*}
g&(x,h):=
p.v.\int_{\R}
\frac{ u_p^{(d)}(x+h t)- (1-|x|^2)^p}{|t|^{1+\alpha}} \,dt\\
&=p.v. \int_{\R}
  \frac{ (1-s^2)^p(1-|x|^2+|h_dx|^2)^p - (1-|x|^2)^p }
    {|-|x|h_d + s \sqrt{|h_dx|^2-|x|^2+1}|^{1+\alpha}}\,
 \sqrt{|h_dx|^2-|x|^2+1}\,ds\\
&=
(1-|x|^2+|h_dx|^2)^{p-\alpha/2}
 p.v.\int_{\R} \frac{  u_p(s) - (1- \frac{|h_dx|^2}{1-|x|^2+|h_dx|^2})^p }
 { |s - \frac{|x|h_d}{\sqrt{1-|x|^2+|h_dx|^2}}|^{1+\alpha}}\,ds\\
&=(1-|x|^2+|h_dx|^2)^{p-\alpha/2} {\mathcal{A}_{1,-\alpha}^{-1}}
\Delta^{\alpha/2} u_p\Big(\frac{|x|h_d}{\sqrt{1-|x|^2+|h_dx|^2}}\Big).
\end{align*}
\end{proof}

\begin{proof}[Proof of formula (\ref{Deltaupn}) of Theorem~\ref{thm:up} for $d>1$]
We have by Lemmata~\ref{lem:Deltaup} and~\ref{lem:Deltaradial},
\begin{align*}
 \Delta^{\alpha/2} u_p^{(d)}(x) &=
 \frac{\Ad  B(-\frac{\alpha}{2}, p+1)}{2} \int_{S^{d-1}} 
\frac{
 {} _2F_1\Big(\frac{\alpha+1}{2}, -p+\frac{\alpha}{2}; \frac{1}{2}; \frac{|x|^2h_d^2}{1- |x|^2 + |h_dx|^2}  \Big)
}{
(1- |x|^2 + |h_dx|^2)^{-p+\alpha/2}
}
\,dh\\
&=: \frac{\Ad  B(-\frac{\alpha}{2}, p+1)}{2} I_{S^{d-1}}.
\end{align*}
We transform the integrand function using \cite[formula 2.9(4), page 105]{Erdelyi},
\begin{align*}
\frac{
 {} _2F_1\Big(\frac{\alpha+1}{2}, -p+\frac{\alpha}{2}; \frac{1}{2}; \frac{|x|^2h_d^2}{1- |x|^2 + |h_dx|^2}  \Big)
}{
(1- |x|^2 + |h_dx|^2)^{-p+\alpha/2}
} &= \frac{
 {} _2F_1\Big(-\frac{\alpha}{2}, -p+\frac{\alpha}{2}; \frac{1}{2}; \frac{|x|^2h_d^2}{|x|^2-1}  \Big)
}{
(1- |x|^2)^{-p+\alpha/2}
}.
\end{align*}
Therefore,
\begin{align*}
I_{S^{d-1}} &=  \int_{S^{d-1}} 
\frac{
 {} _2F_1\Big(-\frac{\alpha}{2}, -p+\frac{\alpha}{2}; \frac{1}{2}; \frac{|x|^2h_d^2}{|x|^2-1}  \Big)
}{
(1- |x|^2)^{-p+\alpha/2}
}\,dh \\
&= \frac{ 2\pi^{\frac{d-1}{2}}(1- |x|^2)^{p-\alpha/2}}{\Gamma(\frac{d-1}{2})}
\int_{-1}^1 
 {} _2F_1\Big(-\frac{\alpha}{2}, -p+\frac{\alpha}{2}; \frac{1}{2}; \frac{|x|^2h^2}{|x|^2-1}  \Big)
(1-h^2)^{\frac{d-3}{2}} \,dh.
\end{align*}
Let
\[
 \phi(z)=\int_{-1}^1 
 {} _2F_1\Big(-\frac{\alpha}{2}, -p+\frac{\alpha}{2}; \frac{1}{2}; \frac{z h^2}{z-1}  \Big)
(1-h^2)^{\frac{d-3}{2}} \,dh, \quad z\in \C, |z|<1.
\]
Since $\re \frac{z h^2}{z-1} < \frac{1}{2}$ for $|z|<1$, the function $\phi$ is analytic
in the unit disc $\{z:|z|<1\}$.
For $|z|<\frac{1}{2}$ we calculate the integral defining $\phi$ by using Taylor's expansion,
\begin{align*}
\phi(z)&=
 \sum_{k=0}^\infty \frac{\Gamma(-\frac{\alpha}{2}+k)}{\Gamma(-\frac{\alpha}{2})}
 \frac{\Gamma(-p+\frac{\alpha}{2}+k)}{\Gamma(-p+\frac{\alpha}{2})}
 \frac{\Gamma(\frac{1}{2})}{\Gamma(\frac{1}{2}+k) k!}
\left( \frac{z}{z-1} \right)^k
\int_{-1}^1 h^{2k}(1-h^2)^{\frac{d-3}{2}}\,dh\\
&= 
\frac{\Gamma(\frac{d-1}{2}) \Gamma(\frac{1}{2}) }{ \Gamma(\frac{d}{2}) }
\sum_{k=0}^\infty \frac{\Gamma(-\frac{\alpha}{2}+k)}{\Gamma(-\frac{\alpha}{2})}
 \frac{\Gamma(-p+\frac{\alpha}{2}+k)}{\Gamma(-p+\frac{\alpha}{2})}
 \frac{\Gamma(\frac{d}{2})}{\Gamma(\frac{d}{2}+k) k!}
\left( \frac{z}{z-1} \right)^k\\
&= 
\frac{\Gamma(\frac{d-1}{2}) \Gamma(\frac{1}{2}) }{ \Gamma(\frac{d}{2}) }
 {} _2F_1\Big(-\frac{\alpha}{2}, -p+\frac{\alpha}{2}; \frac{d}{2}; \frac{z}{z-1}  \Big)\\
&=\
\frac{\Gamma(\frac{d-1}{2}) \Gamma(\frac{1}{2}) }{ \Gamma(\frac{d}{2}) }
 {} _2F_1\Big(\frac{\alpha+d}{2}, -p+\frac{\alpha}{2}; \frac{d}{2}; z  \Big) (1-z)^{-p+\alpha/2}
=:\psi(z).
\end{align*}
In the last line we have used \cite[formula 2.9(4), page 105]{Erdelyi} again.
The functions $\phi$ and $\psi$ are both analytic in the unit disc,  hence $\phi(z)=\psi(z)$
for all $|z|<1$. We put $z=|x|^2$ and the proof is finished.
\end{proof}

\begin{lem}\label{lem:DeltaradialV}
Let $d\geq 2$, $0<\alpha<2$ and $p>-1$. If $x\in \R^d$ and $|x|<1$, then
\begin{align}\label{eq:DeltaradialV}
 \Delta^{\alpha/2} v_p^{(d)}(x) &=  x_d \Delta^{\alpha/2} u_p^{(d)}(x)\\
&+
 \frac{\Ad}{2\Aj} \int_{S^{d-1}} T^{p-\alpha/2} h_d
\left(
    T^{1/2} \Delta^{\alpha/2} v_p
-\langle h,x\rangle \Delta^{\alpha/2} u_p \right)
\bigg(  \frac{\langle h,x\rangle}{\sqrt{T}}  \bigg) \, dh, \nonumber
\end{align}
where $T = T(x,h)=1- |x|^2 + \langle h,x\rangle^2$.
\end{lem}

\begin{proof}
We have for $|x|<1$,
\begin{align*}
 \Delta^{\alpha/2}v_p^{(d)}(x)  &= \Ad p.v. \int_{\R^d} \frac{v_p^{(d)}(y) - v_p^{(d)}(x)}
     {|x-y|^{d+\alpha}}\,dy\\
&=\frac{\Ad}{2} \int_{S^{d-1}} dh
\; p.v.\int_{\R}
\frac{ v_p^{(d)}(x+h t) - v_p^{(d)}(x) }{|t|^{1+\alpha}} \,dt.
\end{align*}
We calculate the inner principal value integral by changing the variable
$t=-\langle h,x\rangle + s \sqrt{T}$. We obtain
\begin{align*}
g(x,h)&:=
p.v.\int_{\R}
\frac{ v_p^{(d)}(x+h t)- v_p^{(d)}(x)}{|t|^{1+\alpha}} \,dt\\
&=p.v. \int_{\R}
  \frac{ (1-s^2)_+^p T^p (x_d-h_d \langle h,x\rangle+h_ds\sqrt{T}) - (1-|x|^2)^px_d }
    {|-\langle h,x\rangle + s \sqrt{T}|^{1+\alpha}}\,
 \sqrt{T}\,ds\\
&=
T^{p-\alpha/2}
 p.v.\int_{\R} \frac{ (1-s^2)_+^p (x_d-h_d\langle h,x\rangle+h_ds\sqrt{T}) - (1- \frac{\langle h,x\rangle^2}{T})^px_d }
    {|s - \frac{\langle h,x\rangle}{\sqrt{T}}|^{1+\alpha}}\,ds\\
&=
T^{p-\alpha/2} (x_d-h_d\langle h,x\rangle)
 p.v.\int_{\R} \frac{ (1-s^2)_+^p  - (1- \frac{\langle h,x\rangle^2}{T})^p }
    {|s - \frac{\langle h,x\rangle}{\sqrt{T}}|^{1+\alpha}}\,ds\\
  &\quad +T^{p+1/2-\alpha/2} h_d
 p.v.\int_{\R} \frac{ (1-s^2)_+^ps  - (1- \frac{\langle h,x\rangle^2}{T})^p \frac{\langle h,x\rangle}{\sqrt{T}} }
    {|s - \frac{\langle h,x\rangle}{\sqrt{T}}|^{1+\alpha}}\,ds\\
&=T^{p-\alpha/2} \frac{x_d-h_d\langle h,x\rangle}{\Aj} \Delta^{\alpha/2} u_p\Big(\frac{\langle h,x\rangle}{\sqrt{T}}\Big)\\
&\quad +
T^{p+1/2-\alpha/2} \frac{h_d}{\Aj} \Delta^{\alpha/2} v_p\Big(\frac{\langle h,x\rangle}{\sqrt{T}}\Big).
\end{align*}
The result follows from Lemma~\ref{lem:Deltaradial}.
\end{proof}

\begin{proof}[Proof of formula (\ref{Deltavpn}) of Theorem~\ref{thm:up} for $d>1$]
We may assume that $x\neq 0$, since for $x=0$ the formula is obvious.
We denote $T = T(x,h)=1- |x|^2 + \langle h,x\rangle^2$.
By Lemmata~\ref{lem:Deltavp} and~\ref{lem:DeltaradialV},
\begin{align*}
 \Delta^{\alpha/2} v_p^{(d)}(x) &=  x_d \Delta^{\alpha/2} u_p^{(d)}(x)
+
 \frac{\Ad  B(-\frac{\alpha}{2}, p+1) }{2}  \times \\
&\qquad \times \int_{S^{d-1}}
 h_d\langle h,x\rangle T^{p-\alpha/2} 
F(x,h)
\,dh,
\end{align*}
where
\[
 F(x,h) =
(\alpha+1) {} _2F_1\Big(\frac{\alpha+3}{2}, -p+\frac{\alpha}{2}; \frac{3}{2}; \frac{\langle h,x\rangle^2}{T}  \Big)
-  _2F_1\Big(\frac{\alpha+1}{2}, -p+\frac{\alpha}{2}; \frac{1}{2}; \frac{\langle h,x\rangle^2}{T}  \Big).
\]
We transform $F(x,h)$ using \cite[formula 2.8(35), page 103]{Erdelyi} and
\cite[formula 2.9(4), page 105]{Erdelyi},
\begin{align*}
F(x,h) &=
\alpha \cdot {} _2F_1\Big(\frac{\alpha+1}{2}, -p+\frac{\alpha}{2}; \frac{3}{2}; \frac{\langle h,x\rangle^2}{T}  \Big)\\
 &=
\alpha
\bigg(\frac{1-|x|^2}{T}\bigg)^{p-\alpha/2}
 {} _2F_1\Big(1 - \frac{\alpha}{2}, -p+\frac{\alpha}{2}; \frac{3}{2}; \frac{\langle h,x\rangle^2}{|x|^2-1}  \Big).
\end{align*}
Hence,
\begin{align*}
\int_{S^{d-1}}&
 h_d\langle h,x\rangle T^{p-\alpha/2} F(x,h)\,dh\\
 &=
\alpha
(1-|x|^2)^{p-\alpha/2} 
\int_{S^{d-1}} 
 {} _2F_1&\Big(1-\frac{\alpha}{2}, -p+\frac{\alpha}{2}; \frac{3}{2}; \frac{\langle h,x\rangle^2}{|x|^2-1}  \Big)  h_d\langle h,x\rangle
\,dh\\
&=:\alpha
(1-|x|^2)^{p-\alpha/2} I_{S^{d-1}}.
\end{align*}
By symmetry, for any $e_1$, $e_2\in \R^d$ with $|e_1|=|e_2|=1$,
\[
 \int_{S^{d-1}} f_1(\langle h,e_1\rangle) f_2(\langle h,e_2\rangle)\,dh =
 \int_{S^{d-1}} f_1(\langle h,e_2\rangle) f_2(\langle h,e_1\rangle)\,dh,
\]
where $f_1$ and $f_2$ are any functions for which the integrals make sense.
Using this observation for $e_1=\frac{x}{|x|}$ and $e_2=(0,\ldots,0,1)$ we obtain
\begin{align*}
I_{S^{d-1}} &=
\int_{S^{d-1}} 
 {} _2F_1\Big(1-\frac{\alpha}{2}, -p+\frac{\alpha}{2}; \frac{3}{2}; \frac{ h_d^2 |x|^2}{|x|^2-1}  \Big)  \langle h,x\rangle h_d
\,dh\\
&=
 \int_{S^{d-1}} 
 {} _2F_1\Big(1-\frac{\alpha}{2}, -p+\frac{\alpha}{2}; \frac{3}{2}; \frac{ h_d^2 |x|^2}{|x|^2-1}  \Big)   h_d^2 x_d
\,dh\\
&= \frac{2\pi^{\frac{d-1}{2}} x_d}{\Gamma(\frac{d-1}{2})}
 \int_{-1}^1
 {} _2F_1\Big(1-\frac{\alpha}{2}, -p+\frac{\alpha}{2}; \frac{3}{2}; \frac{ h^2 |x|^2}{|x|^2-1}  \Big)   h^2 (1-h^2)^{\frac{d-3}{2}}\,dh.
\end{align*}
Let
\[
 \phi(z) = \int_{-1}^1
 {} _2F_1\Big(1-\frac{\alpha}{2}, -p+\frac{\alpha}{2}; \frac{3}{2}; \frac{ h^2 z}{z-1}  \Big)   h^2 (1-h^2)^{\frac{d-3}{2}}\,dh,
\quad z\in \C, |z|<1.
\]
Similarly as in the proof of formula (\ref{Deltaupn}) we observe that $\phi$ is analytic in the unit disc,
and calculate $\phi(z)$ for $|z|<\frac{1}{2}$ by using Taylor's expansion,
\begin{align*} 
\phi(z) &=\int_{-1}^1 \sum_{k=0}^\infty \frac{ (1-\frac{\alpha}{2})_k (-p+\frac{\alpha}{2})_k \Gamma(\frac{3}{2}) }
  {\Gamma(k+\frac{3}{2}) k!} \bigg(\frac{z}{z-1}\bigg)^k h^{2k+2}(1-h^2)^{\frac{d-3}{2}}\,dh\\
&=\frac{  \Gamma(\frac{3}{2}) \Gamma(\frac{d-1}{2})}{\Gamma(\frac{d}{2}+1)}
\sum_{k=0}^\infty  \frac{ (1-\frac{\alpha}{2})_k (-p+\frac{\alpha}{2})_k \Gamma(\frac{d}{2}+1) }{ \Gamma(k+\frac{d}{2}+1) k!}
\bigg(\frac{z}{z-1}\bigg)^k\,dh\\
&=\frac{  \Gamma(\frac{3}{2}) \Gamma(\frac{d-1}{2})}{\Gamma(\frac{d}{2}+1)}
 {} _2F_1\Big(1-\frac{\alpha}{2}, -p+\frac{\alpha}{2}; 1+\frac{d}{2}; \frac{z}{z-1}  \Big).
\end{align*}
Since the function in the last line is analytic in the unit disc (note that $\re \frac{z}{z-1}<\frac{1}{2}$ if $|z|<1$),
we conclude that
\begin{align*}
I_{S^{d-1}} &= \frac{\pi^{\frac{d}{2}} x_d}{\Gamma(\frac{d}{2}+1)}\;
 {} _2F_1\Big(1-\frac{\alpha}{2}, -p+\frac{\alpha}{2}; 1+\frac{d}{2}; \frac{|x|^2}{|x|^2-1}  \Big)\\
&= \frac{\pi^{d/2}x_d}{\Gamma(1+\frac{d}{2})}
\; {} _2F_1\Big(\frac{\alpha+d}{2}, -p+\frac{\alpha}{2}; 1+\frac{d}{2}; |x|^2  \Big)
(1-|x|^2)^{-p+\alpha/2}.
\end{align*}
In the last line we have used \cite[formula 2.9(4), page 105]{Erdelyi}.
By \cite[formula 2.8(35), page 103]{Erdelyi} we get
\begin{align*}
 \Delta^{\alpha/2} v_p^{(d)}(x) &= 
x_d  \frac{\Ad\pi^{d/2}   B(-\frac{\alpha}{2}, p+1)  }{\Gamma(d/2)}
\bigg( 
  \; {} _2F_1\Big(\frac{\alpha+d}{2}, -p+\frac{\alpha}{2}; \frac{d}{2}; |x|^2  \Big)\\
&\qquad\qquad\qquad\qquad\qquad\quad
+ \frac{\alpha}{d}  \; {} _2F_1\Big(\frac{\alpha+d}{2}, -p+\frac{\alpha}{2}; 1+\frac{d}{2}; |x|^2  \Big)
\bigg)\\
&=x_d  \frac{\Ad\pi^{d/2}   B(-\frac{\alpha}{2}, p+1) (\alpha+d) }{d\Gamma(d/2)} \times\\
&\qquad\times
 \; {} _2F_1\Big(\frac{\alpha+d+2}{2}, -p+\frac{\alpha}{2}; 1+\frac{d}{2}; |x|^2  \Big),
\end{align*}
and the proof is finished.
\end{proof}

In Table~\ref{table:valuesd} we list the fractional Laplacian for some power functions in $\R^d$.
\begin{table}[ht]
\caption{Values of fractional Laplacian for some functions vanishing outside the unit ball in $\R^d$.}
\centering
\begin{tabular}{c c}
\hline\hline
$f(x)$ in the ball & $\Delta^{\alpha/2} f(x)$ in the ball \\ [0.5ex]
\hline
$(1-|x|^2)^{\alpha/2}$ & $-2^\alpha\Gamma(\frac{\alpha}{2}+1) \Gamma(\frac{d+\alpha}{2}) \Gamma(\frac{d}{2})^{-1}$ \\
$(1-|x|^2)^{1+\alpha/2}$ & $-2^\alpha\Gamma(\frac{\alpha}{2}+2) \Gamma(\frac{d+\alpha}{2})\Gamma(\frac{d}{2})^{-1}
  \left(1 - (1+\frac{\alpha}{d})|x|^2\right)$\\
\hline
$(1-|x|^2)^{\alpha/2}x_d$ & $-2^\alpha\Gamma(\frac{\alpha}{2}+1) \Gamma(\frac{d+\alpha}{2}+1) \Gamma(\frac{d}{2}+1)^{-1}x_d$ \\
$(1-|x|^2)^{1+\alpha/2}x_d$ & $-2^\alpha\Gamma(\frac{\alpha}{2}+2) \Gamma(\frac{d+\alpha}{2}+1)\Gamma(\frac{d}{2}+1)^{-1}
  \left(1 - (1+\frac{\alpha}{d+2})|x|^2\right)$\\
\hline\hline
\end{tabular}
\label{table:valuesd}
\end{table} 

\section{Lower bounds for eigenvalues}\label{sec:lower}
Our approach to lower  bounds is similar to that of \cite{KBBD-bc}.
The method is to use a suitable superharmonic function $\nu$ and
the resulting Hardy inequality with the weight given by the Fitzsimmons' ratio
${-\Delta^{\alpha/2} \nu }/{\nu}$ (\cite{Fitz}).
To estimate $\lambda_1$, the following calculation will be used.

\begin{lem}\label{lem:updest}
Let $\eta_{d,\alpha}=(6d-4+(4-d)\alpha-\alpha^2)/(\alpha+2)^2$ and
\[
 \psi(x) =      (1-|x|^2)_+^{\alpha/2} + \eta_{d,\alpha} (1-|x|^2)_+^{1+\alpha/2},\quad x\in \R^d.
\]
Then
\[
 \frac{-\Delta^{\alpha/2} \psi(x)}{\psi(x)} \geq \mu_{d,\alpha},
\quad |x|<1,
\]
where $\mu_{d,\alpha}$ is defined in (\ref{eq:mu}).
\end{lem}
\begin{proof}
We have by Table~\ref{table:valuesd}
\[
 \frac{-\Delta^{\alpha/2} \psi(x)}{\psi(x)} = 
\frac{2^\alpha \Gamma(\frac{\alpha}{2}+1) \Gamma(\frac{\alpha+d}{2}) }{ \Gamma(\frac{d}{2})(1-|x|^2)^{\alpha/2}}
   \frac{1 + \eta \frac{\alpha+2}{2}(1-(1+\frac{\alpha}{d})|x|^2)}{1+\eta(1-|x|^2)} =: f(|x|^2).
\]
After elementary but tedious calculations we obtain
\[
 f'(x)=\frac{2^\alpha \Gamma(\frac{\alpha}{2}+1) \Gamma(\frac{\alpha+d}{2}) \alpha(\alpha+d)(\alpha+2) \eta^2}
{ 4d\Gamma(\frac{d}{2})  (1-x)^{1+\alpha/2}(1+\eta(1-x))^2 }
 \left(x-\frac{(4-d)\alpha + 6d-8}{6d-4 + (4-d)\alpha-\alpha^2}\right)^2
\geq 0
\]
for $x\in [0,1)$. Hence
\[
 \frac{-\Delta^{\alpha/2} \psi(x)}{\psi(x)} \geq  \frac{-\Delta^{\alpha/2} \psi(0)}{\psi(0)}
=
 \frac{2^\alpha \Gamma(\frac{\alpha}{2}+1) \Gamma(\frac{\alpha+d}{2}) (\alpha+2)(\alpha+d) (6-\alpha)}
         { \Gamma(\frac{d}{2})(12d+(16-2d)\alpha)}.
\]
\end{proof}

We will also consider the ratio ${-\Delta^{\alpha/2}\nu}/{\nu}$ for a~suitable antisymmetric function $\nu$.
We should note that $\nu$ changes sign and so it cannot be called superharmonic.
However, the Fitzsimmons ratio ${-\Delta^{\alpha/2}\nu}/{\nu}$ will prove nonnegative, and we will
obtain a Hardy inequality with the weight ${-\Delta^{\alpha/2}\nu}/{\nu}$.
To this end we start with the following simple lemma.

\begin{lem}\label{lem:utov}
We have
\begin{equation}\label{eq:utov}
 \int_{B_1^{(d)}} x_d^2 \phi(|x|)\,dx = \frac{1}{\pi}  \int_{B_1^{(d+2)}} \phi(|x|)\,dx,
\end{equation}
for any function $\phi$ for which the integrals are absolutely convergent.
\end{lem}
\begin{proof}
We recall the notation $S^{(d-1)} = \{x\in \R^d:|x|=1\}$ for the unit sphere in $\R^d$, and the formula
for its area, $|S^{(d-1)}|=\frac{2\pi^{d/2}}{\Gamma(d/2)}$. We have
\begin{align*}
 \int_{B_1^{(d)}} x_d^2 \phi(|x|)\,dx &=
 \frac{1}{d}  \int_{B_1^{(d)}} |x|^2 \phi(|x|)\,dx
= \frac{2\pi^{d/2}}{d\Gamma(d/2)}  \int_0^1 r^{d+1} \phi(r)\,dr\\
&= \frac{|S^{(d+1)}|}{\pi}  \int_0^1 r^{d+1} \phi(r)\,dr
= \frac{1}{\pi}   \int_{B_1^{(d+2)}} \phi(|x|)\,dx.
\end{align*}
\end{proof}

\begin{proof}[Proof of Proposition~\ref{prop:lambdaeq}]
We consider a linear subspace $R_{d+2} \subset L^2(B_{d+2})$ consisting of all
radial functions, and a linear subspace $A_d = \{ f: f(x)=x_d g(x) \text{ for some radial } g\in L^2(B_d)\}$
of $L^2(B_d)$. Let $T$ be an operator defined by
\[
 (Tf)(x_1,\ldots, x_d) = \sqrt{\pi} x_d f(x_1,\ldots, x_d,0,0), \quad f\in R_{d+2}\cap C(B_{d+2}).
\]
By Lemma~\ref{lem:utov} we obtain that $\|Tf\|_{L^2(B_d)} =  \|f\|_{L^2(B_{d+2})}$,
therefore $T$ may be extended to an isometry from $R_{d+2}$ onto $A_d$.

Let $G$ be the Green operator, i.e., a bounded operator on $\{f\in L^2(\R^d): f=0 \text{ on } B_d^c \}$
defined by $G_d \phi_n = \frac{1}{\lambda_n} \phi_n$. This operator is the inverse of $-\Delta^{\alpha/2}$
(understood as a~generator).
We observe that the following diagram commutes.

\[
\begin{tikzpicture}[]
\node (a) {$R_{d+2}\ni f$};
\node (b) [right=of a, xshift=1.0em] {$Tf\in A_d$};
\node (c) [below=of a, xshift=1.0em, yshift=-1.2em] {$G_{d+2}f$};
\node (d) [right=of c] {$G_d Tf$};
\draw ($(a.south) + (+1.5em,0)$) 
      edge[->] node[auto,swap,font=\scriptsize]{$G_{d+2}$}
      ($(c.north) + (+0.5em,0)$); 
\draw ($(a.east) + (0,0.2em)$) 
      edge[->] node[auto,font=\scriptsize]{$T$}
      ($(b.west) + (0,0.2em)$); 
\draw ($(a.east) + (0,-0.2em)$) 
      edge[<-] node[auto,swap,font=\scriptsize]{$T^{-1}$}
      ($(b.west) + (0,-0.2em)$);
\draw ($(c.east) + (0,0.2em)$) 
      edge[->] node[auto,font=\scriptsize]{$T$}
      ($(d.west) + (0,0.2em)$); 
\draw ($(c.east) + (0,-0.2em)$) 
      edge[<-] node[auto,swap,font=\scriptsize]{$T^{-1}$}
      ($(d.west) + (0,-0.2em)$);
\draw ($(b.south) + (-1.2em,0)$) 
      edge[->] node[auto,font=\scriptsize]{$G_{d}$}
      ($(d.north) + (0,0)$); 
\end{tikzpicture}
\]

Indeed, by \eqref{Deltaupn} and \eqref{Deltavpn} it commutes for functions
$f(x)=(1-|x|^2)_+^n$, where $x\in B_{d+2}$ and $n=0,1,2,\ldots$.
The linear span of the set of those functions is dense in $R_{d+2}$, hence
by boundedness of $G_d$, $G_{d+2}$, $T$ and $T^{-1}$ the diagram commutes
for all $f\in R_{d+2}$.

Therefore we obtain a one to one correspondence between the radial eigenfunctions
of $G_{d+2}$ (or $\Delta^{\alpha/2}$) in $B_{d+2}$ and the $x_d$-antisymmetric eigenfunctions of $G_d$,
moreover, the corresponding eigenvalues are the same. In particular, $\lambda_1^{(d+2)} = \lambda_*^{(d)}$.
\end{proof}

We note that for functions $u$ in the domain of the generator $\Delta^{\alpha/2}$ we have
\begin{equation}\label{eq:Egen}
 \mathcal{E}(u) = -\int_B u \Delta^{\alpha/2}u\, dx,
\end{equation}
in particular (\ref{eq:Egen}) holds for $u(x)=(1-|x|^2)_+^p$ with $p\geq \alpha/2$.
We are now ready to prove our estimates of $\lambda_1$ and $\lambda_*$.

\begin{proof}[Proof of Corollary~\ref{cor:est}]
Let $\psi$ be as in Lemma~\ref{lem:updest}. From that lemma, (\ref{eq:Egen})
 and \cite[Lemma~2.2]{DydaHardy1dim} (or \cite[Proposition~2.3]{MR2469027})
 we obtain
\[
 \mathcal{E}(u) \geq \int_B u^2(x) \frac{-\Delta^{\alpha/2} \psi(x)}{\psi(x)}\,dx
\geq \mu_{d,\alpha} \int_B u^2(x)\,dx,
\]
and hence $\lambda_1 \geq \mu_{d,\alpha}$.
The second part follows now from Proposition~\ref{prop:lambdaeq}.
\end{proof}

\begin{table}[!htb]
\caption{Lower and upper bounds for $\lambda_1$ and $\lambda_*$.
Lower bounds of Corollary~\ref{cor:est} are in the top row.
In the middle row, in italic, there are the upper bounds obtained from (\ref{eq:var}), by plugging in
a certain linear combination of functions
$u_{j+\alpha/2}$  with $j=0,1,\ldots,12$.
The (near optimal) linear combinations were found numerically.
In the bottom row, we give upper bounds obtained by considering
functions $u$ given by (\ref{eq:ff}) and $\eta$ given by (\ref{eq:eta}).
}
\centering
\begin{tabular}{c|c|c|c|c|c}
\hline\hline
\multirow{2}{*}{$\alpha$} &
\multirow{2}{*}{$\lambda_1$ for $d=1$} &
\multirow{2}{*}{$\lambda_1$ for $d=2$} &
$\lambda_1$ for $d=3$  & $\lambda_1$ for $d=4$  &$\lambda_1$ for $d=5$ \\
&&& $\lambda_*$ for $d=1$ & $\lambda_*$ for $d=2$ & $\lambda_*$ for $d=3$ \\
\hline
\multirow{3}{*}{0.1} &
0.9676 & 1.04874 & 1.08633 & 1.1102 & 1.12756 \\
& \emph{0.97261} & \emph{1.05096} & \emph{1.09221} & \emph{1.12082} & \emph{1.14301} \\
& 0.97273 & 1.05103 & 1.09225 & 1.12093 & 1.14327 \\

\hline
\multirow{3}{*}{0.2} &
0.94993 & 1.10549 & 1.18391 & 1.23565 & 1.27419 \\
& \emph{0.95747} & \emph{1.10993} & \emph{1.19655} & \emph{1.25903} & \emph{1.30877} \\
& 0.95764 & 1.11001 & 1.19663 & 1.25927 & 1.30934 \\

\hline
\multirow{3}{*}{0.5} &
0.96202 & 1.3313 & 1.56035 & 1.72814 & 1.86169 \\
& \emph{0.97017} & \emph{1.34374} & \emph{1.60155} & \emph{1.80843} & \emph{1.98572} \\
& 0.97029 & 1.3438 & 1.60173 & 1.8092 & 1.98766 \\

\hline
\multirow{3}{*}{1} &
1.15384 & 1.96349 & 2.60869 & 3.15561 & 3.63636 \\
& \emph{1.15778} & \emph{2.00612} & \emph{2.75476} & \emph{3.45334} & \emph{4.12131} \\
& 1.1578 & 2.00618 & 2.75548 & 3.45616 & 4.12824 \\

\hline
\multirow{3}{*}{1.5} &
1.58614 & 3.13569 & 4.61848 & 6.03622 & 7.39626 \\
& \emph{1.59751} & \emph{3.27594} & \emph{5.05977} & \emph{6.94732} & \emph{8.93319} \\
& 1.59751 & 3.27624 & 5.06201 & 6.95522 & 8.95256 \\

\hline
\multirow{3}{*}{1.8} &
2.01395 & 4.28394 & 6.65946 & 9.07867 & 11.51297 \\
& \emph{2.04874} & \emph{4.56719} & \emph{7.50312} & \emph{10.82218} & \emph{14.49989} \\
& 2.04876 & 4.56781 & 7.50715 & 10.83601 & 14.53414 \\

\hline
\multirow{3}{*}{1.9} &
2.19524 & 4.77496 & 7.54923 & 10.43088 & 13.37504 \\
& \emph{2.24406} & \emph{5.13213} & \emph{8.59576} & \emph{12.5934} & \emph{17.09653} \\
& 2.24409 & 5.1329 & 8.60059 & 12.60997 & 17.13776 \\

\hline\hline
\end{tabular}
\label{table:upper}
\end{table} 

\section{Upper bounds for eigenvalues}\label{sec:upper}
For functions $u\in D(\mathcal{E})$, by the variational formula, we have
\begin{equation}\label{eq:var}
 \lambda_1 \leq \frac{\mathcal{E}(u)}{\int_B u^2\,dx}.
\end{equation}
For $u$ being a linear combinations of functions $u_{j+\alpha/2}^{(d)}$,
it is easy to compute the right hand side of (\ref{eq:var})
by using (\ref{eq:Egen}), Theorem~\ref{thm:up} and the formula
\[
 \int_B |x|^s (1-|x|^2)^t \,dx = \frac{\pi^{d/2} \Gamma( \frac{s+d}{2}) \Gamma( t+1)}
{\Gamma( \frac{d}{2}) \Gamma( \frac{s+d}{2} + t+1)},\quad s>-d,\;\; t>-1.
\]
In particular, for functions of the form
\begin{equation}\label{eq:ff}
 u(x) = (1-|x|^2)_+^{\alpha/2} + \eta (1-|x|^2)_+^{1+\alpha/2},
\end{equation}
we can  explicitly find $\eta$ minimising the right hand side of (\ref{eq:var}).
A calculation yields,
\begin{equation}\label{eq:eta}
\eta_{min} = \frac{\sqrt{w} + d^2+2d-2a^2-6a-4 }{ 4a^2+12a+8},
\end{equation}
where
\begin{align*}
 w &= d^4+4ad^3+8d^3 +8a^2d^2 + 32ad^2 + 28d^2 + 8a^3d+48a^2d\\
& + 88ad + 48d + 4a^4 + 24a^3 + 52a^2 + 48a+16.
\end{align*}

It is possible to further improve the estimates by taking more functions
$u_{j+\alpha/2}^{(d)}$ to define~$u$.
Then, however, the ratio (\ref{eq:var}) should be minimised numerically.

In  Table~\ref{table:upper}, we give bounds for $\lambda_1$ (and hence also for $\lambda_*$)
 obtained by (\ref{eq:var}) and Corollary~\ref{cor:est}.

\subsection*{Acknowledgements}
The author wishes to thank Tadeusz Kulczycki for helpful discussions, Krzysztof Bogdan for
careful reading of the manuscript and numerous comments, and Mateusz Kwa\'snicki for
fruitful discussions and, in particular, suggesting the proof of Proposition~\ref{prop:lambdaeq}.

This research was partially supported by MNiSW grant N N201 397137, 
and by the DFG through SFB-701 'Spectral Structures and Topological Methods in Mathematics'.

\def\cprime{$'$}

\end{document}